\documentclass[12pt]{amsart}

\usepackage[T1]{fontenc}
\usepackage[utf8]{inputenc}
\usepackage{amsmath,amssymb,amsthm,mathtools}
\usepackage[margin=1.1in]{geometry}
\usepackage{enumitem}
\usepackage{graphicx}
\usepackage{booktabs}
\usepackage{xcolor}
\usepackage{microtype}
\usepackage{hyperref}

\hypersetup{
    colorlinks=true,
    linkcolor=blue!60!black,
    citecolor=blue!60!black,
    urlcolor=blue!60!black
}

\theoremstyle{plain}
\newtheorem{theorem}{Theorem}[section]
\newtheorem{proposition}[theorem]{Proposition}
\newtheorem{lemma}[theorem]{Lemma}
\newtheorem{corollary}[theorem]{Corollary}

\theoremstyle{definition}
\newtheorem{definition}[theorem]{Definition}

\theoremstyle{remark}
\newtheorem{remark}[theorem]{Remark}

\newcommand{\R}{\mathbb R}
\newcommand{\N}{\mathbb N}

\newcommand{\safeincludegraphics}[2][0.75\textwidth]{%
  \IfFileExists{#2}{%
    \includegraphics[width=#1]{#2}%
  }{%
    \fbox{\parbox[c][2.0in][c]{0.82\textwidth}{\centering
      Figure file not found: \texttt{\detokenize{#2}}}}%
  }%
}

\title[The 1D $(p,q)$-Laplacian Fu\v{c}ik Spectrum]{
A Time-Map Description of the One-Dimensional
\texorpdfstring{$(p,q)$}{(p,q)}-Laplacian Fu\v{c}ik Spectrum
}

\author{Jingfan Zheng} \author{Adolfo Rumbos}
\date{}

\begin{document}
\begin{abstract}
We study a one-dimensional Fu\v{c}ik-type boundary value problem (BVP) driven by the nonhomogeneous, one-dimensional
$(p,q)$-Laplacian,
\[
    -\left(|u'|^{p-2}u' + |u'|^{q-2}u'\right)'
    =
    \lambda (u^+)^{p-1}
    -
    \mu (u^-)^{q-1},
    \qquad
    u(0)=u(L)=0,
    \qquad L>0,
\]
where $1<q<p$ and $\lambda,\mu\in\R$. The problem is asymmetric: the
positive component is paired with the $p$-power and the negative component
with the $q$-power. This asymmetry prevents the standard homogeneous
$p$-Laplacian scaling argument from producing explicit Fu\v{c}ik curves.

We first prove a quadrant reduction theorem. Any positive nodal interval for a solution
forces $\lambda>0$, and any negative nodal interval forces $\mu>0$.
Hence, all sign-changing solutions occur for values of $(\lambda,\mu)$
in the positive quadrant. Outside the positive quadrant, only one-sign
solutions can occur.

We also formulate the equation as an autonomous planar system and establish
existence and uniqueness for its initial value problem in the positive quadrant.
We then impose the second Dirichlet condition through a shooting formulation
and show that every zero of a nontrivial Dirichlet solution is simple and
isolated. This provides the dynamical foundation for the subsequent
construction and gluing of nodal humps of solutions.

Inside the positive quadrant, we use a classical one-dimensional time-map
method to obtain a projected-geometry description of the spectrum, or the set of values of $(\lambda, \mu)$ for which there exist nontrivial solutions. Every
nodal solution has a common energy $E>0$ shared by all positive and negative
humps. If a solution has $m$ positive humps and $n$ negative humps, then its
parameters are characterized by
\[
    2m\lambda^{-1/p}A(E)
    +
    2n\mu^{-1/q}B(E)
    =L
\]
for some $E>0$, where $A$ and $B$ are parameter-free time maps. We prove
that $A$ is a strictly decreasing function of $E$, $B$ is strictly increasing, and the
sign-changing solutions of the BVP occur in the threshold regions
\[
    \Pi_{m,n}(L)
    =
    \{(\lambda,\mu):M_{m,n}(\lambda,\mu)\le L\}.
\]
Their boundaries are nonincreasing functions
$\mu=\mu^*_{m,n}(\lambda)$ with asymptotes
\[
    \lambda=\left(\frac{m\pi_p}{L}\right)^p,
    \qquad
    \mu=\left(\frac{n\pi_q}{L}\right)^q,
\]
where, for $r>1$,
\[
    \pi_r
    :=
    2(r-1)^{1/r}\int_0^1(1-t^r)^{-1/r}\,dt.
\]

\end{abstract}

\maketitle

\section{Introduction}

The Fu\v{c}ik spectrum was introduced to study boundary value problems
with jumping nonlinearities \cite{Fucik1976}. In its classical form, one
studies equations whose behavior differs on their positive and negative
parts. The subject is closely related to the Ambrosetti--Prodi problem and
to nonlinear resonance with jumping nonlinearities; see, for instance,
\cite{AmbrosettiProdi1972,Dancer1977,Dancer1981,Rynne2000}.

For the one-dimensional homogeneous $p$-Laplacian, the Dirichlet problem
\[
    -(|u'|^{p-2}u')'
    =
    \alpha (u^+)^{p-1}
    -
    \beta (u^-)^{p-1},
    \qquad
    u(0)=u(L)=0,
    \qquad L>0,
\]
has a particularly transparent nodal description. The reason is scaling:
once the sign and parameter are fixed, all humps have a length independent
of their amplitude. Thus nodal solutions are obtained by gluing positive
and negative copies of homogeneous $p$-trigonometric humps. This structure
is closely connected to the generalized trigonometric functions studied in
\cite{Lindqvist1995}. The beginning and structure of the $p$-Laplacian
Fu\v{c}ik spectrum have been studied extensively; see
\cite{CuestaFigueiredoGossez1999,DancerPerera2001,Perera2002,Perera2004}.
Related periodic problems at resonance on the Dancer--Fu\v{c}ik spectrum
were studied by Bliss, Buerger, and Rumbos in
\cite{BlissBuergerRumbos2011}.

For a nontrivial solution $u$, a \emph{nodal interval} is a connected
component of
\[
    \{x\in(0,L):u(x)\ne0\}.
\]
On each nodal interval the solution has a constant sign. When the zeros
are ordered, the nodal intervals are precisely the intervals between
successive zeros, with the boundary points $0$ and $L$ regarded as zeros
because of the Dirichlet boundary conditions.

This paper studies the corresponding one-dimensional question for the
nonhomogeneous $(p,q)$-Laplacian when $p\ne q$. Problems involving the
$(p,q)$-Laplacian have been actively studied in nonlinear analysis,
especially in relation to nonhomogeneous operators, concave--convex
nonlinearities, and nonlinear eigenvalue problems
\cite{MaranoMosconi2018,FariaMiyagakiMotreanuTanaka2014,
MaranoMosconiPapageorgiou2018,MaranoMarinoPapageorgiou2019,
PapageorgiouRadulescuRepovs2019}.

For $r>1$, define
\[
    \phi_r(s):=|s|^{r-2}s,\qquad\text{for } s \in \mathbb R
\]

We consider
\begin{equation}\label{eq:model}
\begin{cases}
-\left(\phi_p(u')+\phi_q(u')\right)'
=
\lambda (u^+)^{p-1}
-
\mu (u^-)^{q-1},
&0<x<L,\\[4pt]
u(0)=u(L)=0,
\end{cases}
\end{equation}
where
\[
    L>0,
    \qquad
    1<q<p,
    \qquad
    \lambda,\mu\in\R.
\]

The operator $-\left(\phi_p(u')+\phi_q(u')\right)'$ contains two different gradient
homogeneities. Consequently, the amplitude of a hump of a solution cannot
be scaled away, and the usual explicit gluing argument for the homogeneous
$p$-Laplacian does not directly apply.

The same asymmetric $(p,q)$-Laplacian Fu\v{c}ik equation was considered by
Djeffal and Benselhoub \cite{DjeffalBenselhoub2023}. Working on bounded
domains, they combined Lusternik--Schnirelmann theory with a constrained
min--max construction to obtain families of continuous decreasing curves
contained in the Fu\v{c}ik spectrum. Their variational approach applies
beyond one space dimension and establishes the existence of selected
spectral families.

The purpose of the present paper is complementary to what has been done
in higher-dimensional settings. Restricting to one space dimension, we
exploit the autonomous ODE structure to classify solutions by their nodal
counts and by a common Hamiltonian energy. In particular, we do not claim
novelty for defining the $(p,q)$-Laplacian Fu\v{c}ik spectrum, nor for the
formal use of a one-dimensional quadrature. Our contribution is a
necessary-and-sufficient time-map description of the set of parameter
values $(\lambda,\mu)$ for which \eqref{eq:model} has a nontrivial
solution. We also obtain explicit boundary formulas and asymptotes for the projected
spectral regions.

The Fu\v{c}ik spectrum associated with \eqref{eq:model} is the set
\[
    \Sigma_{p,q}(L)
    :=
    \left\{
        (\lambda,\mu)\in\R^2:
        \eqref{eq:model}\text{ has a nontrivial weak solution}
    \right\}.
\]

The loss of homogeneity does not destroy the one-dimensional nodal
structure. Instead, it changes the geometry of the regions of the
Fu\v{c}ik spectrum. In the homogeneous $p$-Laplacian problem, the
amplitude of every nodal hump of a solution can be scaled away. So, fixing
the signs and nodal counts produces an equation directly involving the
Fu\v{c}ik parameters only. The resulting components of the Fu\v{c}ik spectrum are curves in the
$(\lambda,\mu)$-plane \cite{CuestaFigueiredoGossez1999,DancerPerera2001}.

For \eqref{eq:model}, the energy of the humps cannot be eliminated. The
natural object is therefore first the energy-lifted relation
\[
    \mathcal L_{m,n}(E;\lambda,\mu)=L
    \qquad
    \text{(total nodal length $=$ interval length)}
\]
in $(E,\lambda,\mu)$-space. The corresponding region in the parameter
plane is its projection,
\[
    \Pi_{m,n}(L)
    =
    \left\{
        (\lambda,\mu)\in(0,\infty)^2:
        \exists E>0,\ 
        \mathcal L_{m,n}(E;\lambda,\mu)=L
    \right\}.
\]

For fixed $(\lambda,\mu)$, define
\[
    M_{m,n}(\lambda,\mu)
    :=
    \min_{E>0}\mathcal L_{m,n}(E;\lambda,\mu),
\]
where the existence of the minimum will be proved below. The projection
then produces the threshold-region condition
\[
    \Pi_{m,n}(L)
    =
    \left\{
        (\lambda,\mu)\in(0,\infty)^2:
        M_{m,n}(\lambda,\mu)\le L
    \right\}.
\]
Thus, the main contribution of this paper is a classification of the
projected spectral geometry.

To explain the construction, define
\[
    \mathcal B(s):=\phi_p(s)+\phi_q(s), \qquad\text{for } s \in \mathbb R
\]
and
\[
    h(\xi)
    :=
    \frac{p-1}{p}\xi^p
    +
    \frac{q-1}{q}\xi^q,
    \qquad \xi\ge0.
\]

If $u$ is positive on a nodal interval, then
\[
    -(\mathcal B(u'))'=\lambda u^{p-1},
\]
and the associated Hamiltonian identity is
\[
    h(|u'|)+\frac{\lambda}{p}u^p=E.
\]
Similarly, if $u<0$ and $v=-u>0$, then
\[
    h(|v'|)+\frac{\mu}{q}v^q=E.
\]
Adjacent humps have the same energy because the solution is $C^1$ and,
at each nodal point, both potential terms vanish. Thus the energy at a
nodal point is determined entirely by the common derivative there.

Before using this identity to construct time maps, we formulate
\eqref{eq:model} as a planar autonomous system in the variables $(u,w)$,
where $w=\mathcal B(u')$. For $\lambda,\mu>0$, we establish well-posedness
of the associated initial value problem. We then impose the Dirichlet
conditions by starting at $u(0)=0$ with a nonzero initial slope and requiring the
resulting trajectory to return to $u=0$ at $x=L$. The uniqueness of the initial value problem
also shows that a nontrivial Dirichlet solution cannot have a double zero;
hence its nodal points are isolated. Only after these steps do we introduce
the time maps and the spectral projection.

This gives the half-hump time maps
\[
    \tau_+(E;\lambda)
    =
    \int_0^{(pE/\lambda)^{1/p}}
    \frac{ds}{\eta\left(E-\frac{\lambda}{p}s^p\right)},
\]
and
\[
    \tau_-(E;\mu)
    =
    \int_0^{(qE/\mu)^{1/q}}
    \frac{ds}{\eta\left(E-\frac{\mu}{q}s^q\right)},
\]
where $\eta=h^{-1}$. After rescaling,
\[
    \tau_+(E;\lambda)=\lambda^{-1/p}A(E),
    \qquad
    \tau_-(E;\mu)=\mu^{-1/q}B(E).
\]

Hence, if a solution has $m$ positive humps and $n$ negative humps, its
length condition is
\[
    2m\lambda^{-1/p}A(E)
    +
    2n\mu^{-1/q}B(E)
    =L.
\]
Accordingly, we define
\[
    \mathcal L_{m,n}(E;\lambda,\mu)
    =
    2m\lambda^{-1/p}A(E)
    +
    2n\mu^{-1/q}B(E),
\]
and $M_{m,n}(\lambda,\mu)$ is the minimum of this function over $E>0$.

We prove that
\[
    \Pi_{m,n}(L)
    =
    \{(\lambda,\mu):M_{m,n}(\lambda,\mu)\le L\}.
\]
The boundary is written as $\mu=\mu^*_{m,n}(\lambda)$, where
\[
    \mu^*_{m,n}(\lambda)
    =
    \inf_{E:\,L>2m\lambda^{-1/p}A(E)}
    \left[
        \frac{2nB(E)}
        {L-2m\lambda^{-1/p}A(E)}
    \right]^q.
\]

The paper is organized as follows.
Section~\ref{sec:scope} gives the quadrant reduction and explains the model.
Section~\ref{sec:preliminaries} establishes the weak formulation and
regularity. Section~\ref{sec:ivp} formulates the associated planar system and
establishes local existence and uniqueness for the initial value problem away
from the equilibrium. Section~\ref{sec:energy} proves the Hamiltonian identity,
completes the well-posedness argument at the equilibrium, and gives global
continuation. Section~\ref{sec:bvp} imposes the second Dirichlet boundary
condition through a shooting formulation and proves that the zeros of every
nontrivial Dirichlet solution are simple and isolated.
Section~\ref{sec:time_maps} constructs individual humps, and
Section~\ref{sec:nodal} gives the nodal reduction.
Section~\ref{sec:asymptotics} studies the parameter-free time maps.
Section~\ref{sec:one_sign} treats the positive and negative one-hump
solutions and compares their spectral geometry with the classical
homogeneous $p$-Laplacian.
Section~\ref{sec:branches} gives the sign-changing spectral
characterization, and Section~\ref{sec:boundary} gives the boundary
formulas and asymptotes.
The final section summarizes the resulting time-map description of the spectrum.

\section{Scope, Quadrants, and Model Choice}
\label{sec:scope}

The parameters $\lambda$ and $\mu$ are allowed to be arbitrary real
numbers. However, sign-changing solutions can occur only for parameter
values in the positive quadrant of the $(\lambda,\mu)$-plane.

\begin{lemma}[Positive humps force $\lambda>0$]
\label{lem:positive_requires_lambda}
Let $u$ be a nontrivial weak solution of \eqref{eq:model}. If $u$ has a
positive nodal interval, then $\lambda>0$.
\end{lemma}

\begin{proof}
Let $J=(a,b)$ be a positive nodal interval. Then
\[
    u(a)=u(b)=0,
    \qquad
    u>0\quad\text{on }(a,b).
\]
By Proposition~\ref{prop:regularity} below, $u\in C^1[0,L]$ and
$w:=\mathcal B(u')\in C^1[0,L]$. On $J$,
\[
    -w'=\lambda u^{p-1}.
\]
Assume that $\lambda\le0$. Then
\[
    w'=-\lambda u^{p-1}\ge0,
\]
so $w$ is nondecreasing. Because $u(a)=0$ and $u>0$ immediately to the
right of $a$, we have $u'(a)\ge0$ and hence $w(a)\ge0$. Similarly,
$u'(b)\le0$ and $w(b)\le0$. A nondecreasing function satisfying
\[
    w(a)\ge0\ge w(b)
\]
must be identically zero on $J$. Therefore
\[
    u'=\mathcal B^{-1}(w)=0
\]
on $J$, contradicting $u>0$ there. Thus $\lambda>0$.
\end{proof}

\begin{lemma}[Negative humps force $\mu>0$]
\label{lem:negative_requires_mu}
Let $u$ be a nontrivial weak solution of \eqref{eq:model}. If $u$ has a
negative nodal interval, then $\mu>0$.
\end{lemma}

\begin{proof}
Let $J=(a,b)$ be a negative nodal interval and put $v=-u>0$ on $J$.
By Proposition~\ref{prop:regularity} below, the flux equation is classical
on $J$. Since $\mathcal B$ is odd, $v$ satisfies
\[
    -(\mathcal B(v'))'
    =
    \mu v^{q-1}
    \qquad\text{on }J.
\]
Applying Lemma~\ref{lem:positive_requires_lambda} to $v$, with $\mu$ in
place of $\lambda$, gives $\mu>0$.
\end{proof}

\begin{corollary}[Quadrant reduction]
\label{cor:quadrant_reduction}
If \eqref{eq:model} has a sign-changing solution, then
\[
    \lambda>0,
    \qquad
    \mu>0.
\]
\end{corollary}

\begin{proof}
A sign-changing solution has at least one positive nodal interval and at least one negative nodal interval. Apply Lemmas~\ref{lem:positive_requires_lambda}
and~\ref{lem:negative_requires_mu}.
\end{proof}

\begin{remark}[One-dimensional scope]
The method used here is intrinsically one-dimensional. It uses a Hamiltonian identity on ordered nodal intervals and glues humps of solutions by matching endpoint slopes. In higher dimensions, nodal sets are not ordered intervals, and the
scalar time-map construction does not directly apply. General regularity theory for quasilinear elliptic equations is available in much greater generality \cite{Lieberman1991}.
\end{remark}

\begin{remark}[The asymmetric model]
More generally, one can consider
\[
    -\left(\phi_p(u')+\phi_q(u')\right)'
    =
    \lambda(u^+)^{r-1}
    -
    \mu(u^-)^{s-1},
    \qquad \text{for }r,s>1.
\]
The corresponding half-hump maps are
\[
    \tau_r(E;\lambda)
    =
    \int_0^{(rE/\lambda)^{1/r}}
    \frac{d\xi}
    {\eta\left(E-\frac{\lambda}{r}\xi^r\right)}
\]
and
\[
    \tau_s(E;\mu)
    =
    \int_0^{(sE/\mu)^{1/s}}
    \frac{d\xi}
    {\eta\left(E-\frac{\mu}{s}\xi^s\right)}.
\]
The present paper focuses on $r=p$ and $s=q$, for which the two time maps
have complementary endpoint behavior.
\end{remark}

\section{Weak Formulation and Regularity}
\label{sec:preliminaries}

Let
\[
    X:=W_0^{1,p}(0,L).
\]
Since $1<q<p$ and $(0,L)$ is bounded,
\[
    X\hookrightarrow W_0^{1,q}(0,L),
    \qquad
    X\hookrightarrow C[0,L].
\]

\begin{definition}[Weak solution]
A function $u\in X$ is a weak solution of \eqref{eq:model} if, for every
$\varphi\in X$,
\begin{equation}\label{eq:weak}
    \int_0^L
    \left(\phi_p(u')+\phi_q(u')\right)\varphi'\,dx
    =
    \int_0^L
    \left[
        \lambda(u^+)^{p-1}
        -
        \mu(u^-)^{q-1}
    \right]\varphi\,dx.
\end{equation}
\end{definition}

Define
\[
    \mathcal B(s)
    :=
    \phi_p(s)+\phi_q(s)
    =
    |s|^{p-2}s+|s|^{q-2}s, \qquad \text{for }s \in \mathbb R
\]

\begin{lemma}[Basic properties of $\mathcal B$]
\label{lem:B_properties}
The map $\mathcal B:\R\to\R$ is continuous, odd, strictly increasing,
and onto. Hence it is a homeomorphism.
\end{lemma}

\begin{proof}
Continuity and oddness are immediate. Both summands are increasing, and
their sum is strictly increasing. Moreover,
\[
    \mathcal B(s)\to\pm\infty
    \qquad\text{as }s\to\pm\infty,
\]
respectively. Thus $\mathcal B$ is a continuous strictly increasing
bijection and therefore has a continuous inverse.
\end{proof}

\begin{proposition}[Regularity and flux equation]
\label{prop:regularity}
Let $u\in X$ be a weak solution of \eqref{eq:model}. Then
\[
    u\in C^1[0,L].
\]
Moreover, if
\[
    w:=\mathcal B(u'),
\]
then $w\in C^1[0,L]$ and
\begin{equation}\label{eq:flux}
    -w'
    =
    \lambda(u^+)^{p-1}
    -
    \mu(u^-)^{q-1}
    \qquad\text{on }[0,L].
\end{equation}
\end{proposition}

\begin{proof}
Since
\[
    u\in X\hookrightarrow C[0,L],
\]
the function
\[
    f_u(x)
    :=
    \lambda(u^+(x))^{p-1}
    -
    \mu(u^-(x))^{q-1}, \ \ \text{for }x \in [0,L]
\]
is continuous. The weak formulation implies
\[
    (\mathcal B(u'))'=-f_u
\]
in the sense of distributions. Hence,
\[
    w=\mathcal B(u')
\]
has a $C^1$ representative satisfying
\[
    w'=-f_u.
\]
Since $\mathcal B^{-1}$ is continuous,
\[
    u'=\mathcal B^{-1}(w)
\]
is continuous, and therefore $u\in C^1[0,L]$.
\end{proof}

\section{The Associated Dynamical System and the Initial Value Problem}
\label{sec:ivp}

The regularity result of Section~\ref{sec:preliminaries} suggests using the
flux as a second phase variable. Set
\[
    w:=\mathcal B(u')
\]
and define
\[
    g_{\lambda,\mu}:\R\to\R,
    \qquad
    g_{\lambda,\mu}(u)
    :=
    \lambda(u^+)^{p-1}
    -
    \mu(u^-)^{q-1}.
\]
Since $\mathcal B$ is a homeomorphism of $\R$ onto $\R$, the differential
equation in \eqref{eq:model} is equivalent to the autonomous planar system
\begin{equation}\label{eq:dynamical_system}
\begin{cases}
    u'=\mathcal B^{-1}(w),\\[3pt]
    w'=-g_{\lambda,\mu}(u).
\end{cases}
\end{equation}
Thus a solution of the differential equation on an interval
$I\subset\R$ may be viewed as a trajectory
\[
    x\longmapsto (u(x),w(x)),
    \qquad x\in I,
\]
in the $(u,w)$-plane. At this stage no boundary condition at $x=L$ is being
imposed; we first study the initial value problem for the system in \eqref{eq:dynamical_system}.

\begin{proposition}[Equivalence with the planar system]
\label{prop:system_equivalence}
Let $I\subset\R$ be an interval. If $u$ is a weak solution of the
differential equation in \eqref{eq:model} on $I$ and
$w=\mathcal B(u')$, then $(u,w)\in C^1(I,\R^2)$ and satisfies
\eqref{eq:dynamical_system}. Conversely, if
$(u,w)\in C^1(I,\R^2)$ satisfies \eqref{eq:dynamical_system}, then
$w=\mathcal B(u')$ and $u$ satisfies the differential equation in
\eqref{eq:model} in the weak sense on $I$.
\end{proposition}

\begin{proof}
The first assertion follows from Proposition~\ref{prop:regularity}. Conversely,
the first equation in \eqref{eq:dynamical_system} gives
\[
    u'=\mathcal B^{-1}(w),
\]
and hence
\[
    w=\mathcal B(u').
\]
The second equation then becomes
\[
    -(\mathcal B(u'))'
    =
    \lambda(u^+)^{p-1}-\mu(u^-)^{q-1}.
\]
Multiplication by a compactly supported test function and integration by parts
give the corresponding weak formulation in $I$.
\end{proof}

For prescribed data
\[
    x_0\in\R,
    \qquad
    (u_0,w_0)\in\R^2,
\]
consider the initial value problem
\begin{equation}\label{eq:ivp}
\begin{cases}
    u'=\mathcal B^{-1}(w),\\[3pt]
    w'=-\lambda(u^+)^{p-1}+\mu(u^-)^{q-1},\\[3pt]
    u(x_0)=u_0,\qquad w(x_0)=w_0.
\end{cases}
\end{equation}

\begin{proposition}[Local existence]
\label{prop:ivp_local_existence}
For every $\lambda,\mu\in\R$ and every initial state $(u_0,w_0)$,
\eqref{eq:ivp} has at least one local $C^1$ solution.
\end{proposition}

\begin{proof}
The vector field
\[
    (u,w)
    \longmapsto
    \left(
        \mathcal B^{-1}(w),
        -\lambda(u^+)^{p-1}+\mu(u^-)^{q-1}
    \right)
\]
is continuous on $\R^2$. The assertion therefore follows from the Peano
existence theorem; see, for example, \cite{Hartman2002}.
\end{proof}

The vector field need not be locally Lipschitz on the coordinate axes for
arbitrary $1<q<p$. Nevertheless, away from the equilibrium $(0,0)$ the
special structure of the system gives local uniqueness.

\begin{proposition}[Local uniqueness away from the equilibrium]
\label{prop:ivp_local_uniqueness_nonzero}
Assume $\lambda,\mu>0$ and $(u_0,w_0)\ne(0,0)$. Then the local solution of
\eqref{eq:ivp} through $(u_0,w_0)$ is unique.
\end{proposition}

\begin{proof}
Suppose first that $w_0\ne0$. Then
\[
    u'(x_0)=\mathcal B^{-1}(w_0)\ne0.
\]
Thus $u$ is strictly monotone near $x_0$ and may be used as the independent
variable. Dividing the two equations in \eqref{eq:dynamical_system} gives
\begin{equation}\label{eq:w_as_function_of_u}
    \frac{dw}{du}
    =
    \frac{-g_{\lambda,\mu}(u)}{\mathcal B^{-1}(w)}.
\end{equation}
Since $w$ stays away from zero near $w_0$, the right-hand side is continuous
in $u$ and locally Lipschitz in the dependent variable $w$. Hence the scalar
initial value problem \eqref{eq:w_as_function_of_u} has a unique local
solution. Once $w$ is known as a function of $u$, the equation
\[
    \frac{dx}{du}
    =
    \frac{1}{\mathcal B^{-1}(w(u))}
\]
uniquely recovers the parametrization by $x$.

Suppose next that $w_0=0$ and $u_0\ne0$. Since $\lambda,\mu>0$,
\[
    g_{\lambda,\mu}(u_0)\ne0,
\]
and therefore
\[
    w'(x_0)=-g_{\lambda,\mu}(u_0)\ne0.
\]
Thus $w$ is strictly monotone near $x_0$ and may be used as the independent
variable. We obtain
\begin{equation}\label{eq:u_as_function_of_w}
    \frac{du}{dw}
    =
    \frac{\mathcal B^{-1}(w)}{-g_{\lambda,\mu}(u)}.
\end{equation}
Because $u$ stays away from zero near $u_0$, the denominator does not vanish,
and the right-hand side is continuous in $w$ and locally Lipschitz in the
dependent variable $u$. Hence \eqref{eq:u_as_function_of_w} has a unique
local solution. The parametrization is then uniquely recovered from
\[
    \frac{dx}{dw}
    =
    \frac{1}{-g_{\lambda,\mu}(u(w))}.
\]
This proves local uniqueness at every nonzero phase point.
\end{proof}

\begin{remark}[The equilibrium case]
\label{rem:equilibrium_deferred}
The preceding argument deliberately leaves the initial state $(0,0)$
untreated. At that point neither $u$ nor $w$ can be used immediately as a
new independent variable. Uniqueness at the equilibrium will follow from the
conserved Hamiltonian established in the next section.
\end{remark}

\section{Hamiltonian Identity and Global Well-Posedness}
\label{sec:energy}

Define
\[
    \Phi(s)
    :=
    \int_0^s\mathcal B(t)\,dt
    =
    \frac{|s|^p}{p}
    +
    \frac{|s|^q}{q}, \qquad\text{for } s \in \mathbb R
\]
and
\[
    H(s)
    :=
    s\mathcal B(s)-\Phi(s)
    =
    \frac{p-1}{p}|s|^p
    +
    \frac{q-1}{q}|s|^q, \qquad\text{for } s \in \mathbb R
\]
For $\xi\ge0$, set
\begin{equation}\label{eq:h_def}
    h(\xi)
    :=
    \frac{p-1}{p}\xi^p
    +
    \frac{q-1}{q}\xi^q.
\end{equation}
Then
\[
    H(s)=h(|s|), \qquad\text{for } s \in \mathbb R
\]
The function $h$ is strictly increasing from $[0,\infty)$ onto
$[0,\infty)$; let
\[
    \eta:=h^{-1}.
\]
Also define
\[
    \widehat H(y)
    :=
    \int_0^y\mathcal B^{-1}(\zeta)\,d\zeta.
\]

\begin{lemma}[Legendre identity]
\label{lem:legendre}
For every $s\in\R$,
\[
    \widehat H(\mathcal B(s))=h(|s|).
\]
\end{lemma}

\begin{proof}
The function $\mathcal B$ is strictly increasing and locally
absolutely continuous. Using the change of variables $y=\mathcal B(t)$ and
then integrating by parts, we obtain for every $s\in\R$, 
\begin{align*}
    \widehat H(\mathcal B(s))
    &=
    \int_0^{\mathcal B(s)}\mathcal B^{-1}(y)\,dy\\
    &=
    \int_0^s t\,\mathcal B'(t)\,dt\\
    &=
    s\mathcal B(s)-\int_0^s\mathcal B(t)\,dt.
\end{align*}
Thus, for all $s\in\mathbb R$,
\[
    \widehat H(\mathcal B(s))=H(s)=h(|s|).
\]
\end{proof}

For $\lambda,\mu>0$, define the phase-space energy
\begin{equation}\label{eq:phase_energy}
    \mathcal E(u,w)
    :=
    \widehat H(w)
    +
    \frac{\lambda}{p}(u^+)^p
    +
    \frac{\mu}{q}(u^-)^q.
\end{equation}

\begin{proposition}[Hamiltonian identity]
\label{prop:global_energy}
Assume $\lambda,\mu>0$ and let $(u,w)$ be a $C^1$ solution of
\eqref{eq:dynamical_system} on an interval $I$. Then
\begin{equation}\label{eq:global_energy}
    \widehat H(w(x))
    +
    \frac{\lambda}{p}(u^+(x))^p
    +
    \frac{\mu}{q}(u^-(x))^q
    =E
\end{equation}
for some constant $E\ge0$ and all $x\in I$. If the trajectory is nontrivial,
then $E>0$.
\end{proposition}

\begin{proof}
Since $\widehat H'(y)=\mathcal B^{-1}(y)$ and
$u'=\mathcal B^{-1}(w)$,
\[
    \frac{d}{dx}\widehat H(w(x))
    =
    \mathcal B^{-1}(w(x))w'(x)
    =u'(x)w'(x).
\]
Because $p,q>1$, the functions $t\mapsto(t^+)^p$ and
$t\mapsto(t^-)^q$ are $C^1$. Therefore
\begin{align*}
\frac{d}{dx}
\left[
    \widehat H(w)
    +\frac{\lambda}{p}(u^+)^p
    +\frac{\mu}{q}(u^-)^q
\right]
&=
    u'w'
    +\lambda(u^+)^{p-1}u'
    -\mu(u^-)^{q-1}u'\\
&=
    u'\left[-\lambda(u^+)^{p-1}+\mu(u^-)^{q-1}\right]
    +\lambda(u^+)^{p-1}u'
    -\mu(u^-)^{q-1}u'\\
&=0.
\end{align*}
Hence the displayed quantity is constant.

By Lemma~\ref{lem:legendre},
\[
    \widehat H(w)
    =
    h\!\left(|\mathcal B^{-1}(w)|\right)
    \ge0,
\]
and the two potential terms are nonnegative. Thus, $E\ge0$. If the trajectory
is nontrivial, then either $u\ne0$ somewhere or $w\ne0$ somewhere, and at
such a point at least one of the three nonnegative terms is strictly positive.
Hence, $E>0$.
\end{proof}

\begin{corollary}[Uniqueness at the equilibrium]
\label{cor:equilibrium_uniqueness}
Assume $\lambda,\mu>0$. The only solution of \eqref{eq:dynamical_system}
that passes through $(u,w)=(0,0)$ is the equilibrium solution
\[
    (u,w)\equiv(0,0).
\]
\end{corollary}

\begin{proof}
A trajectory through $(0,0)$ has energy $E=0$. Every term in
\eqref{eq:phase_energy} is nonnegative, and their sum vanishes only at the
origin. Conservation of energy therefore forces $(u(x),w(x))=(0,0)$
throughout its interval of existence.
\end{proof}

\begin{theorem}[Global existence and uniqueness for the initial value problem]
\label{thm:ivp_wellposed}
Assume $\lambda,\mu>0$. For every
\[
    x_0\in\R,
    \qquad
    (u_0,w_0)\in\R^2,
\]
the initial value problem \eqref{eq:ivp} has a unique global solution
\[
    (u,w)\in C^1(\R;\R^2).
\]
Moreover,
\[
    \mathcal E(u(x),w(x))
    =
    \mathcal E(u_0,w_0)
    \qquad\text{for every }x\in\R.
\]
where $\mathcal E$ is defined in \eqref{eq:phase_energy}.
\end{theorem}

\begin{proof}
Local existence follows from Proposition~\ref{prop:ivp_local_existence}.
Local uniqueness at nonzero states follows from
Proposition~\ref{prop:ivp_local_uniqueness_nonzero}, while uniqueness at the
origin follows from Corollary~\ref{cor:equilibrium_uniqueness}.

It remains to prove global continuation. Along every local solution,
Proposition~\ref{prop:global_energy} gives
\[
    \mathcal E(u(x),w(x))=E_0,
    \qquad
    E_0:=\mathcal E(u_0,w_0).
\]
The energy is coercive in the phase variables. Indeed,
\[
    \widehat H(w)
    =
    h\!\left(|\mathcal B^{-1}(w)|\right)
    \longrightarrow+\infty
    \qquad\text{as }|w|\to\infty,
\]
and, since $\lambda,\mu>0$,
\[
    \frac{\lambda}{p}(u^+)^p
    +
    \frac{\mu}{q}(u^-)^q
    \longrightarrow+\infty
    \qquad\text{as }|u|\to\infty.
\]
Thus the trajectory remains in the compact energy level
\[
    \{(u,w)\in\R^2:\mathcal E(u,w)=E_0\}.
\]
The standard continuation theorem for ordinary differential equations then
extends the solution to all $x\in\R$; see \cite{Hartman2002}. Local
uniqueness patches these continuations into a unique global solution.
\end{proof}

\begin{corollary}[Uniqueness from position and slope]
\label{cor:ivp_slope}
Assume $\lambda,\mu>0$. Given
\[
    x_0\in\R,
    \qquad
    u(x_0)=u_0,
    \qquad
    u'(x_0)=v_0,
\]
there exists a unique global solution of the differential equation. It is the
solution of the problem in Theorem~\ref{thm:ivp_wellposed} with
\[
    w_0=\mathcal B(v_0).
\]
\end{corollary}

\begin{remark}[Why Picard--Lindel\"of is not invoked directly]
\label{rem:ivp_lipschitz}
For arbitrary $1<q<p$, the vector field in \eqref{eq:dynamical_system} is
continuous but need not be locally Lipschitz on the coordinate axes. The
argument above uses the special one-dimensional structure instead: away from
the axes one phase variable can be used as the independent variable, and the
conserved Hamiltonian handles the equilibrium. This avoids imposing
additional restrictions on $p$ and $q$ merely to invoke the standard
Picard--Lindel\"of theorem for the planar system.
\end{remark}

\section{The Dirichlet Boundary Condition and Isolation of Zeros}
\label{sec:bvp}

The initial value problem determines a trajectory once its starting state is
specified. We now impose the two Dirichlet conditions and identify which of
those trajectories are solutions of the boundary value problem
\eqref{eq:model}.

Fix $\lambda,\mu>0$. For each initial slope $a\in\R$, let $u_a$ denote the
unique global solution given by Corollary~\ref{cor:ivp_slope} with
\begin{equation}\label{eq:shooting_initial_data}
    u_a(0)=0,
    \qquad
    u_a'(0)=a.
\end{equation}
Equivalently, the initial flux is $w_a(0)=\mathcal B(a)$. Define the shooting
map
\begin{equation}\label{eq:shooting_map}
    S_{\lambda,\mu}(a):=u_a(L).
\end{equation}
The first Dirichlet condition has been built into the initial data; the second
condition is exactly the scalar equation
\[
    S_{\lambda,\mu}(a)=0.
\]

\begin{proposition}[IVP--BVP correspondence]
\label{prop:shooting_correspondence}
Assume $\lambda,\mu>0$.
\begin{enumerate}[label=\textup{(\roman*)}]
\item The initial slope $a=0$ produces only the trivial solution.
\item If $a\ne0$, then $u_a|_{[0,L]}$ is a nontrivial weak solution of
\eqref{eq:model} if and only if
\[
    S_{\lambda,\mu}(a)=0.
\]
\item Every nontrivial weak solution of \eqref{eq:model} is obtained in this
way from the unique initial slope $a=u'(0)\ne0$.
\end{enumerate}
\end{proposition}

\begin{proof}
If $a=0$, then the initial phase state is $(0,0)$, so
Corollary~\ref{cor:equilibrium_uniqueness} gives $u_0\equiv0$.

For $a\ne0$, the differential equation is already satisfied by construction,
and $u_a(0)=0$. Hence the restriction to $[0,L]$ satisfies both Dirichlet
conditions exactly when $u_a(L)=0$, that is, when
$S_{\lambda,\mu}(a)=0$. Since $a\ne0$, such a solution is nontrivial.

Conversely, let $u$ be a nontrivial weak solution of \eqref{eq:model}.
Proposition~\ref{prop:regularity} gives $u\in C^1[0,L]$, and
Proposition~\ref{prop:system_equivalence} identifies it with a trajectory of
\eqref{eq:dynamical_system}. Taking $a=u'(0)$, uniqueness of the initial
value problem implies $u=u_a$ on $[0,L]$. Moreover $a\ne0$, for otherwise
the initial state would be $(0,0)$ and the solution would be trivial. Finally,
the boundary condition $u(L)=0$ gives $S_{\lambda,\mu}(a)=0$.
\end{proof}

The preceding proposition isolates the role of the second boundary condition:
solving the Dirichlet problem is equivalent to finding nonzero initial slopes
whose trajectories return to the $u=0$ axis precisely at $x=L$. We next
justify the nodal decomposition of such trajectories.

\begin{corollary}[Simple and isolated zeros]
\label{cor:simple_zeros}
Assume $\lambda,\mu>0$ and let $u\not\equiv0$ be a weak solution of
\eqref{eq:model}. If
\[
    u(x_0)=0,
    \qquad x_0\in[0,L],
\]
then
\[
    u'(x_0)\ne0.
\]
Consequently, every zero of $u$ is isolated.
\end{corollary}

\begin{proof}
Suppose that $u(x_0)=u'(x_0)=0$. Then the associated phase state is
\[
    (u(x_0),w(x_0))
    =
    (0,\mathcal B(0))
    =(0,0).
\]
By uniqueness for the initial value problem,
Theorem~\ref{thm:ivp_wellposed}, the trajectory through this state is the
zero trajectory. Hence $u\equiv0$, contradicting the hypothesis. Therefore
$u'(x_0)\ne0$. Since $u'$ is continuous, $u$ changes sign monotonically in a
sufficiently small one-sided or two-sided neighborhood of $x_0$, and the zero
is isolated.
\end{proof}

\begin{corollary}[Common slope magnitude at nodal points]
\label{cor:common_nodal_slope}
Under the assumptions of Corollary~\ref{cor:simple_zeros}, let $E>0$ be the
Hamiltonian energy of $u$. If $u(x_0)=0$, then
\[
    h(|u'(x_0)|)=E,
    \qquad
    |u'(x_0)|=\eta(E).
\]
In particular, every nodal crossing of the same solution has the same slope
magnitude.
\end{corollary}

\begin{proof}
At a zero the two potential terms in \eqref{eq:global_energy} vanish, so
\[
    E=\widehat H(w(x_0)).
\]
Since $w(x_0)=\mathcal B(u'(x_0))$, Lemma~\ref{lem:legendre} yields
\[
    E=h(|u'(x_0)|).
\]
Applying $\eta=h^{-1}$ gives the second identity.
\end{proof}

\begin{corollary}[Finitely many zeros]
\label{cor:finitely_many_zeros}
Under the assumptions of Corollary~\ref{cor:simple_zeros}, the zero set of
$u$ in $[0,L]$ is finite.
\end{corollary}

\begin{proof}
If the zero set were infinite, compactness of $[0,L]$ would give an
accumulation point $x_*$. By continuity, $u(x_*)=0$. Choosing distinct zeros
$x_k\to x_*$ gives
\[
    \frac{u(x_k)-u(x_*)}{x_k-x_*}=0.
\]
Since $u$ is differentiable at $x_*$, passing to the limit yields
$u'(x_*)=0$, contradicting Corollary~\ref{cor:simple_zeros}.
\end{proof}

\section{Time Maps and Single Humps}
\label{sec:time_maps}

The previous sections separate the two issues that precede the time-map construction. The initial value problem determines a unique global trajectory, while Proposition~\ref{prop:shooting_correspondence} identifies the additional
return condition required for that trajectory to solve the Dirichlet boundary
value problem. Corollary~\ref{cor:simple_zeros} shows that every zero of a
nontrivial Dirichlet solution is a genuine transverse crossing. We may
therefore analyze one nodal excursion at a time and compute the travel time
from a zero to its turning point.

For $m>1$, $\alpha>0$, and $E>0$, define
\[
    a_m(E;\alpha):=\left(\frac{mE}{\alpha}\right)^{1/m},
\]
and
\begin{equation}\label{eq:generic_tau}
    \tau_m(E;\alpha)
    :=
    \int_0^{a_m(E;\alpha)}
    \frac{ds}{\eta\left(E-\frac{\alpha}{m}s^m\right)}.
\end{equation}

\begin{lemma}[Finiteness of $\tau_m$]
\label{lem:tau_finite}
For every $m>1$, $\alpha>0$, and $E>0$, the integral
\eqref{eq:generic_tau} is finite.
\end{lemma}

\begin{proof}
Only the upper endpoint can be singular. Near $a=a_m(E;\alpha)$,
\[
    E-\frac{\alpha}{m}s^m
    =\frac{\alpha}{m}(a^m-s^m)
    \sim C(a-s).
\]
By the small-energy asymptotics proved below, $\eta(y)\sim C_0y^{1/q}$ as $y\downarrow0$. Hence the integrand is bounded by $C(a-s)^{-1/q}$ near $a$, which is integrable because $q>1$.
\end{proof}

\begin{proposition}[Half-hump construction]
\label{prop:half_hump}
Let $m>1$, $\alpha>0$, and $E>0$. Put
\[
    a=a_m(E;\alpha),
    \qquad
    \tau=\tau_m(E;\alpha).
\]
There exists a unique increasing function $U:[0,\tau]\to[0,a]$ satisfying
\[
    U(0)=0,
    \quad U(\tau)=a,
    \quad U'>0\text{ on }(0,\tau),
\]
and
\[
    h(U')+\frac{\alpha}{m}U^m=E.
\]
It is defined implicitly by
\[
    x=\int_0^{U(x)}
    \frac{ds}{\eta\left(E-\frac{\alpha}{m}s^m\right)}.
\]
Moreover,
\[
    -(\mathcal B(U'))'=\alpha U^{m-1}
    \qquad\text{on }(0,\tau).
\]
\end{proposition}

\begin{proof}
Define
\[
    F(y):=\int_0^y
    \frac{ds}{\eta\left(E-\frac{\alpha}{m}s^m\right)},
    \qquad 0\le y\le a.
\]
Then $F$ is continuous and strictly increasing, with $F(0)=0$ and
$F(a)=\tau<\infty$. Set $U=F^{-1}$. Differentiation gives
\[
    U'=\eta\left(E-\frac{\alpha}{m}U^m\right),
\]
and therefore
\[
    h(U')+\frac{\alpha}{m}U^m=E.
\]
Differentiating on $(0,\tau)$ and using
$h'(\xi)=\xi\mathcal B'(\xi)$ for $\xi>0$, we obtain
\[
    (\mathcal B(U'))'+\alpha U^{m-1}=0.
\]
Uniqueness follows from the strict invertibility of $F$.
\end{proof}

Define
\begin{equation}\label{eq:tau_plus}
    \tau_+(E;\lambda)
    :=\tau_p(E;\lambda)
    =\int_0^{(pE/\lambda)^{1/p}}
    \frac{ds}{\eta\left(E-\frac{\lambda}{p}s^p\right)},
\end{equation}
and
\begin{equation}\label{eq:tau_minus}
    \tau_-(E;\mu)
    :=\tau_q(E;\mu)
    =\int_0^{(qE/\mu)^{1/q}}
    \frac{ds}{\eta\left(E-\frac{\mu}{q}s^q\right)}.
\end{equation}

\begin{proposition}[Unique full humps]
\label{prop:full_humps}
Let $E>0$.
\begin{enumerate}[label=\textup{(\roman*)}]
\item For every $\lambda>0$, there exists a unique positive full hump
$P_{E,\lambda}$ on $[0,2\tau_+(E;\lambda)]$ satisfying
\[
    -(\mathcal B(P_{E,\lambda}'))'=\lambda P_{E,\lambda}^{p-1},
\]
\[
    P_{E,\lambda}(0)=P_{E,\lambda}(2\tau_+(E;\lambda))=0,
    \qquad P_{E,\lambda}>0\text{ in the interior},
\]
and
\[
    P_{E,\lambda}'(0)=\eta(E),
    \qquad
    P_{E,\lambda}'(2\tau_+(E;\lambda))=-\eta(E).
\]
\item For every $\mu>0$, there exists a unique negative full hump
$N_{E,\mu}$ on $[0,2\tau_-(E;\mu)]$ satisfying
\[
    -(\mathcal B(N_{E,\mu}'))'=-\mu(N_{E,\mu}^-)^{q-1},
\]
\[
    N_{E,\mu}(0)=N_{E,\mu}(2\tau_-(E;\mu))=0,
    \qquad N_{E,\mu}<0\text{ in the interior},
\]
and
\[
    N_{E,\mu}'(0)=-\eta(E),
    \qquad
    N_{E,\mu}'(2\tau_-(E;\mu))=\eta(E).
\]
\end{enumerate}
\end{proposition}

\begin{proof}
For the positive hump, Proposition~\ref{prop:half_hump} gives the unique
increasing half-hump from the zero level to the turning point
\[
    a_+=\left(\frac{pE}{\lambda}\right)^{1/p}.
\]
At the turning point the phase state is $(u,w)=(a_+,0)$. Reflecting the
increasing half-hump across this point produces a $C^1$ solution of the
system. By Theorem~\ref{thm:ivp_wellposed}, there is only one trajectory
through $(a_+,0)$, so this reflected curve is the unique continuation past
the maximum.

At either zero endpoint, Corollary~\ref{cor:common_nodal_slope} gives
$|u'|=\eta(E)$, and the orientation of the hump determines the displayed
signs of the endpoint derivatives. Conversely, any positive full hump of
energy $E$ has strictly decreasing flux on its interior, hence a unique
turning point; its increasing half is the half-hump from
Proposition~\ref{prop:half_hump}, and IVP uniqueness fixes the continuation.

For the negative hump, apply the same argument to $v=-u>0$ with parameter
$\mu$ and exponent $q$. Oddness of $\mathcal B$ gives the stated negative
profile and endpoint derivatives.
\end{proof}

\section{Nodal Reduction}
\label{sec:nodal}

\begin{definition}[Nodal interval]
A nodal interval of $u$ is a connected component of $\{x\in(0,L):u(x)\ne0\}$. It is positive or negative according to the sign of $u$ on that interval.
\end{definition}

\begin{theorem}[Nodal reduction theorem]
\label{thm:nodal_reduction}
Assume $\lambda,\mu>0$. Let $u\not\equiv0$ be a weak solution of
\eqref{eq:model}. Then $u$ has finitely many nodal intervals, all sharing a common energy $E>0$. Every positive nodal interval has length
$2\tau_+(E;\lambda)$, and every negative nodal interval has length
$2\tau_-(E;\mu)$.
\end{theorem}

\begin{proof}
Finiteness follows from Corollary~\ref{cor:finitely_many_zeros}. The global identity \eqref{eq:global_energy} supplies one constant $E>0$ on all of $[0,L]$.

Let $(a,b)$ be a positive nodal interval. By
Corollary~\ref{cor:common_nodal_slope},
\[
    h(|u'(a)|)=h(|u'(b)|)=E,
    \qquad
    |u'(a)|=|u'(b)|=\eta(E).
\]
The signs imply $u'(a)=\eta(E)$ and $u'(b)=-\eta(E)$. If
$w=\mathcal B(u')$, then on $(a,b)$,
\[
    w'=-\lambda u^{p-1}<0.
\]
Thus $w$ crosses zero exactly once, at a unique point $c\in(a,b)$. Hence $u$
is strictly increasing on $(a,c)$ and strictly decreasing on $(c,b)$. The energy identity becomes
\[
    h(|u'|)+\frac{\lambda}{p}u^p=E,
\]
and at the maximum
\[
    u(c)=\left(\frac{pE}{\lambda}\right)^{1/p}.
\]
The initial state at any point of this nodal interval determines the local
trajectory uniquely by Theorem~\ref{thm:ivp_wellposed}. Thus the restriction
of $u$ to $(a,b)$ is the uniquely determined positive hump of energy $E$
with these endpoint data. Separating variables on each half gives
\[
    c-a=b-c=\tau_+(E;\lambda),
\]
so $b-a=2\tau_+(E;\lambda)$.

For a negative nodal interval, put $v=-u>0$ and repeat the argument with
$\mu$ and exponent $q$ to obtain length $2\tau_-(E;\mu)$.
\end{proof}

\section{Parameter-Free Time Maps}
\label{sec:asymptotics}

Rescaling \eqref{eq:tau_plus}--\eqref{eq:tau_minus} yields
\[
    \tau_+(E;\lambda)=\lambda^{-1/p}A(E),
    \qquad
    \tau_-(E;\mu)=\mu^{-1/q}B(E),
\]
where
\begin{equation}\label{eq:A_def}
    A(E):=p^{1/p}E^{1/p}
    \int_0^1\frac{dt}{\eta(E(1-t^p))},
\end{equation}
and
\begin{equation}\label{eq:B_def}
    B(E):=q^{1/q}E^{1/q}
    \int_0^1\frac{dt}{\eta(E(1-t^q))}.
\end{equation}
For $r>1$, define
\begin{equation}\label{eq:pi_r}
    \pi_r:=2(r-1)^{1/r}\int_0^1(1-t^r)^{-1/r}\,dt
    =\frac{2\pi(r-1)^{1/r}}{r\sin(\pi/r)}.
\end{equation}

\begin{lemma}[Asymptotics of $\eta$]
\label{lem:eta_asymptotics}
As $y\downarrow0$,
\[
    \eta(y)\sim\left(\frac{q}{q-1}y\right)^{1/q},
\]
and as $y\to\infty$,
\[
    \eta(y)\sim\left(\frac{p}{p-1}y\right)^{1/p}.
\]
\end{lemma}

\begin{proof}
These follow immediately from
\[
    h(\xi)\sim\frac{q-1}{q}\xi^q\quad(\xi\downarrow0),
    \qquad
    h(\xi)\sim\frac{p-1}{p}\xi^p\quad(\xi\to\infty),
\]
by inversion.
\end{proof}

For $m>1$, define
\[
    R_m(y):=\frac{y^{1/m}}{\eta(y)},
    \qquad y>0.
\]
Then
\[
    A(E)=p^{1/p}\int_0^1
    (1-t^p)^{-1/p}R_p(E(1-t^p))\,dt,
\]
and
\[
    B(E)=q^{1/q}\int_0^1
    (1-t^q)^{-1/q}R_q(E(1-t^q))\,dt.
\]

\begin{lemma}[Limits and domination]
\label{lem:R_limits}
The functions $R_p$ and $R_q$ satisfy
\[
    \lim_{y\to\infty}R_p(y)=\left(\frac{p-1}{p}\right)^{1/p},
    \qquad
    \lim_{y\downarrow0}R_q(y)=\left(\frac{q-1}{q}\right)^{1/q}.
\]
Moreover, the integrands above admit integrable dominators in the regimes $E\ge1$ for $A$ and $0<E\le1$ for $B$.
\end{lemma}

\begin{proof}
The limits follow from Lemma~\ref{lem:eta_asymptotics}. If $y=h(\xi)$, then
\[
    R_p(h(\xi))
    =\left(\frac{p-1}{p}+\frac{q-1}{q}\xi^{q-p}\right)^{1/p},
\]
which is strictly decreasing, while
\[
    R_q(h(\xi))
    =\left(\frac{p-1}{p}\xi^{p-q}+\frac{q-1}{q}\right)^{1/q},
\]
which is strictly increasing. These monotonicities give the stated domination. Near $t=1$, the first dominator behaves like $(1-t^p)^{-1/q}$ and the second like $(1-t^q)^{-1/q}$, both integrable because $q>1$.
\end{proof}

\begin{proposition}[Continuity of the parameter-free time maps]
\label{prop:AB_continuous}
The functions $A,B:(0,\infty)\to(0,\infty)$ are finite and continuous.
Consequently, for every fixed $(m,n)\in\mathbb N^2$, the map
\[
    (E,\lambda,\mu)
    \longmapsto
    2m\lambda^{-1/p}A(E)+2n\mu^{-1/q}B(E)
\]
is continuous on $(0,\infty)^3$.
\end{proposition}

\begin{proof}
Finiteness follows from Lemma~\ref{lem:tau_finite}, because
$A(E)=\tau_+(E;1)$ and $B(E)=\tau_-(E;1)$. We prove continuity using the representations involving $R_p$ and $R_q$.

Fix $E_0>0$ and let $E_j\to E_0$. For all sufficiently large $j$, the energies $E_j$ lie in a compact interval $[a,b]\subset(0,\infty)$. Since $R_p$ is strictly decreasing,
\[
    R_p(E_j(1-t^p))
    \le R_p(a(1-t^p))
    \qquad (0\le t<1),
\]
and the resulting majorant is integrable because its integral is a constant multiple of $A(a)$. Pointwise convergence follows from continuity of $\eta$. The dominated convergence theorem therefore gives $A(E_j)\to A(E_0)$.

Similarly, the strict increase of $R_q$ gives
\[
    R_q(E_j(1-t^q))
    \le R_q(b(1-t^q)),
\]
and the corresponding majorant is integrable because its integral is a constant multiple of $B(b)$. Dominated convergence gives $B(E_j)\to B(E_0)$. The joint continuity of the displayed length expression is then immediate.
\end{proof}

\begin{proposition}[Monotonicity of $A$ and $B$]
\label{prop:AB_monotone}
The function $A$ is strictly decreasing on $(0,\infty)$, and $B$ is strictly increasing on $(0,\infty)$.
\end{proposition}

\begin{proof}
The representation above and the strict decrease of $R_p$ show that
$A(E_1)>A(E_2)$ when $0<E_1<E_2$. The strict increase of $R_q$ similarly gives $B(E_1)<B(E_2)$.
\end{proof}

\begin{proposition}[Endpoint behavior]
\label{prop:AB_limits}
The following limits hold:
\[
    \lim_{E\downarrow0}A(E)=+\infty,
    \qquad
    \lim_{E\to\infty}A(E)=\frac{\pi_p}{2},
\]
and
\[
    \lim_{E\downarrow0}B(E)=\frac{\pi_q}{2},
    \qquad
    \lim_{E\to\infty}B(E)=+\infty.
\]
\end{proposition}

\begin{proof}
For $A(E)$ as $E\to\infty$, apply dominated convergence using
Lemma~\ref{lem:R_limits}; this gives
\[
    \lim_{E\to\infty}A(E)
    =(p-1)^{1/p}\int_0^1(1-t^p)^{-1/p}\,dt
    =\frac{\pi_p}{2}.
\]
As $E\downarrow0$, $R_p(E(1-t^p))\to\infty$ pointwise, so Fatou's lemma gives $A(E)\to\infty$.

For $B(E)$ as $E\downarrow0$, dominated convergence gives
\[
    \lim_{E\downarrow0}B(E)
    =(q-1)^{1/q}\int_0^1(1-t^q)^{-1/q}\,dt
    =\frac{\pi_q}{2}.
\]
As $E\to\infty$, $R_q(E(1-t^q))\to\infty$ pointwise, and Fatou's lemma gives $B(E)\to\infty$.
\end{proof}

\section{One-Sign Solutions and the Trivial Spectral Regions}
\label{sec:one_sign}

The sign-changing analysis involves nodal counts
$(m,n)\in\mathcal I_{\mathrm{sc}}\subset\N^2$, so both $m$ and $n$ are
positive. The two remaining cases are the one-sign solutions. In terms of
nodal counts, these correspond formally to
\[
    (m,n)=(1,0)
    \qquad\text{and}\qquad
    (m,n)=(0,1).
\]
They play the role of the trivial branches in the classical Fu\v{c}ik
spectrum.

We first note that a nontrivial one-sign solution consists of a single
nodal interval.

\begin{lemma}[One-sign solutions have one nodal interval]
\label{lem:one_sign_single_interval}
If $u\ge0$, $u\not\equiv0$, is a weak solution of \eqref{eq:model}, then
\[
    u>0\qquad\text{on }(0,L).
\]
Similarly, if $u\le0$, $u\not\equiv0$, then
\[
    u<0\qquad\text{on }(0,L).
\]
\end{lemma}

\begin{proof}
Suppose first that $u\ge0$ and $u\not\equiv0$. Since $u$ has a positive
nodal interval, Lemma~\ref{lem:positive_requires_lambda} gives
$\lambda>0$. Because $u^-\equiv0$, the equation on the support of $u$
reduces to
\[
    -(\mathcal B(u'))'=\lambda u^{p-1}.
\]
The same calculation as in Proposition~\ref{prop:global_energy} gives
the energy identity
\[
    h(|u'|)+\frac{\lambda}{p}u^p=E
\]
on every positive nodal interval.

If $u(x_0)=0$ for some $x_0\in(0,L)$, then $x_0$ is a local minimum, so
$u'(x_0)=0$. Hence the energy on an adjacent positive nodal interval is
\[
    E
    =
    h(|u'(x_0)|)
    +
    \frac{\lambda}{p}u(x_0)^p
    =0.
\]
Both terms in the energy identity are nonnegative, so this forces
$u=u'=0$ throughout that nodal interval, a contradiction. Therefore
$u>0$ on $(0,L)$.

The nonpositive case follows in the same way after setting $v=-u>0$.
\end{proof}

\begin{proposition}[Positive one-hump branch]
\label{prop:positive_one_hump}
Fix $\mu\in\R$. Problem~\eqref{eq:model} has a strictly positive solution
if and only if
\[
    \lambda>
    \left(\frac{\pi_p}{L}\right)^p.
\]
For every such $\lambda$, there is a unique energy $E_+(\lambda)>0$
satisfying
\begin{equation}\label{eq:positive_one_hump_length}
    2\lambda^{-1/p}A(E_+(\lambda))=L,
\end{equation}
and the corresponding positive solution consists of one full hump.
\end{proposition}

\begin{proof}
Suppose that $u$ is a strictly positive solution. By
Lemma~\ref{lem:one_sign_single_interval}, $(0,L)$ is its only nodal
interval. Since $u^-=0$, the parameter $\mu$ does not appear in the
equation, which becomes
\[
    -(\mathcal B(u'))'=\lambda u^{p-1}.
\]
Lemma~\ref{lem:positive_requires_lambda} gives $\lambda>0$.

Let $E>0$ denote the energy of the hump. Separating variables as in the
positive-hump part of the proof of Theorem~\ref{thm:nodal_reduction}
shows that its total length is
\[
    2\tau_+(E;\lambda).
\]
Since the hump fills the whole interval $(0,L)$,
\[
    2\tau_+(E;\lambda)=L.
\]
Using
\[
    \tau_+(E;\lambda)=\lambda^{-1/p}A(E),
\]
we obtain
\[
    2\lambda^{-1/p}A(E)=L,
\]
or equivalently,
\begin{equation}\label{eq:A_positive_condition}
    A(E)=\frac{L}{2}\lambda^{1/p}.
\end{equation}

By Propositions~\ref{prop:AB_monotone} and~\ref{prop:AB_limits}, the
function $A$ is continuous and strictly decreasing, with
\[
    \lim_{E\downarrow0}A(E)=+\infty,
    \qquad
    \lim_{E\to\infty}A(E)=\frac{\pi_p}{2}.
\]
Thus $A$ maps $(0,\infty)$ bijectively onto
\[
    \left(\frac{\pi_p}{2},\infty\right).
\]
Consequently, \eqref{eq:A_positive_condition} has a unique solution
$E>0$ if and only if
\[
    \frac{L}{2}\lambda^{1/p}>\frac{\pi_p}{2},
\]
which is equivalent to
\[
    \lambda>
    \left(\frac{\pi_p}{L}\right)^p.
\]

Conversely, if this strict inequality holds, there is a unique
$E_+(\lambda)>0$ satisfying
\eqref{eq:positive_one_hump_length}. Proposition~\ref{prop:full_humps}
then produces the positive full hump of energy $E_+(\lambda)$, and its
length is exactly $L$. Since its negative part vanishes, it satisfies
\eqref{eq:model} for every $\mu\in\R$.
\end{proof}

\begin{proposition}[Negative one-hump branch]
\label{prop:negative_one_hump}
Fix $\lambda\in\R$. Problem~\eqref{eq:model} has a strictly negative
solution if and only if
\[
    \mu>
    \left(\frac{\pi_q}{L}\right)^q.
\]
For every such $\mu$, there is a unique energy $E_-(\mu)>0$ satisfying
\begin{equation}\label{eq:negative_one_hump_length}
    2\mu^{-1/q}B(E_-(\mu))=L,
\end{equation}
and the corresponding negative solution consists of one full hump.
\end{proposition}

\begin{proof}
Suppose that $u$ is a strictly negative solution and set
\[
    v:=-u>0.
\]
Then $u^+=0$, and since $\mathcal B$ is odd, the differential equation
becomes
\[
    -(\mathcal B(v'))'=\mu v^{q-1}.
\]
Lemma~\ref{lem:negative_requires_mu} gives $\mu>0$.

By Lemma~\ref{lem:one_sign_single_interval}, $(0,L)$ is the only nodal
interval. If $E>0$ denotes its energy, the negative-hump time map gives
\[
    2\tau_-(E;\mu)=L.
\]
Using
\[
    \tau_-(E;\mu)=\mu^{-1/q}B(E),
\]
we obtain
\[
    2\mu^{-1/q}B(E)=L,
\]
or equivalently,
\begin{equation}\label{eq:B_negative_condition}
    B(E)=\frac{L}{2}\mu^{1/q}.
\end{equation}

By Propositions~\ref{prop:AB_monotone} and~\ref{prop:AB_limits}, $B$ is
continuous and strictly increasing, with
\[
    \lim_{E\downarrow0}B(E)=\frac{\pi_q}{2},
    \qquad
    \lim_{E\to\infty}B(E)=+\infty.
\]
Hence $B$ maps $(0,\infty)$ bijectively onto
\[
    \left(\frac{\pi_q}{2},\infty\right).
\]
Therefore \eqref{eq:B_negative_condition} has a unique solution $E>0$
if and only if
\[
    \frac{L}{2}\mu^{1/q}>\frac{\pi_q}{2},
\]
or equivalently,
\[
    \mu>
    \left(\frac{\pi_q}{L}\right)^q.
\]

Conversely, whenever this inequality holds, there is a unique
$E_-(\mu)>0$ satisfying \eqref{eq:negative_one_hump_length}.
Proposition~\ref{prop:full_humps} gives the corresponding negative full
hump of length $L$. Since its positive part vanishes, it satisfies
\eqref{eq:model} for every $\lambda\in\R$.
\end{proof}

The preceding propositions identify the two one-sign parts of the
spectrum. Define
\[
    \mathcal P_+(L)
    :=
    \left\{
        (\lambda,\mu)\in\R^2:
        \lambda>
        \left(\frac{\pi_p}{L}\right)^p
    \right\},
\]
and
\[
    \mathcal P_-(L)
    :=
    \left\{
        (\lambda,\mu)\in\R^2:
        \mu>
        \left(\frac{\pi_q}{L}\right)^q
    \right\}.
\]
Then $\mathcal P_+(L)$ is precisely the set of parameters for which a
strictly positive solution exists, while $\mathcal P_-(L)$ is precisely
the set of parameters for which a strictly negative solution exists.

In particular, within the positive quadrant,
\[
    \mathcal P_+(L)\cap(0,\infty)^2
    =
    \left(
        \left(\frac{\pi_p}{L}\right)^p,\infty
    \right)
    \times(0,\infty),
\]
whereas
\[
    \mathcal P_-(L)\cap(0,\infty)^2
    =
    (0,\infty)
    \times
    \left(
        \left(\frac{\pi_q}{L}\right)^q,\infty
    \right).
\]
Thus the one-sign portions of the spectrum are regions rather than single
vertical or horizontal lines.

\begin{remark}[Comparison with the homogeneous $p$-Laplacian]
\label{rem:homogeneous_comparison}
For the classical homogeneous $p$-Laplacian Fu\v{c}ik problem
\[
    -(\phi_p(u'))'
    =
    \alpha(u^+)^{p-1}
    -
    \beta(u^-)^{p-1},
    \qquad
    u(0)=u(L)=0,
\]
the homogeneity of the operator removes the amplitude, or equivalently the
energy, from the nodal-length condition. In particular, a positive full hump
has length
\[
    \pi_p\alpha^{-1/p},
\]
so a strictly positive Dirichlet solution exists only when
\[
    \alpha=\left(\frac{\pi_p}{L}\right)^p.
\]
Since $\beta$ does not enter the equation for a positive solution, this gives
the vertical line
\[
    \alpha=\left(\frac{\pi_p}{L}\right)^p.
\]
Similarly, the negative one-sign branch is the horizontal line
\[
    \beta=\left(\frac{\pi_p}{L}\right)^p.
\]

For \eqref{eq:model}, the situation is different because the operator
$\mathcal B=\phi_p+\phi_q$ is not homogeneous. The positive one-hump length
is
\[
    2\lambda^{-1/p}A(E),
\]
and the negative one-hump length is
\[
    2\mu^{-1/q}B(E),
\]
so the energy cannot be scaled away. By
Proposition~\ref{prop:AB_limits},
\[
    A((0,\infty))
    =
    \left(\frac{\pi_p}{2},\infty\right),
    \qquad
    B((0,\infty))
    =
    \left(\frac{\pi_q}{2},\infty\right).
\]
Consequently, the one-sign portions of the spectrum are the regions
\[
    \mathcal P_+(L)
    =
    \left\{
        (\lambda,\mu)\in\R^2:
        \lambda>
        \left(\frac{\pi_p}{L}\right)^p
    \right\},
\]
and
\[
    \mathcal P_-(L)
    =
    \left\{
        (\lambda,\mu)\in\R^2:
        \mu>
        \left(\frac{\pi_q}{L}\right)^q
    \right\}.
\]
Thus the vertical and horizontal trivial lines of the homogeneous
$p$-Laplacian are replaced here by vertical and horizontal spectral regions.
\end{remark}

\section{Spectral Regions}
\label{sec:branches}

A sign-changing solution alternates signs. Throughout this section,
$\N:=\{1,2,3,\ldots\}$. If a sign-changing solution has $m$ positive humps and $n$ negative
humps, then $|m-n|\le1$. Define
\[
    \mathcal I_{\mathrm{sc}}
    :=
    \{(m,n)\in\N^2:|m-n|\le1\}.
\]
For $(m,n)\in\mathcal I_{\mathrm{sc}}$, define
\begin{equation}\label{eq:length_function}
    \mathcal L_{m,n}(E;\lambda,\mu)
    :=
    2m\lambda^{-1/p}A(E)
    +
    2n\mu^{-1/q}B(E),
\end{equation}
and
\[
    M_{m,n}(\lambda,\mu)
    :=
    \min_{E>0}\mathcal L_{m,n}(E;\lambda,\mu).
\]

\begin{proposition}[Existence of the minimum]
\label{prop:M_exists}
For every $(m,n)\in\mathcal I_{\mathrm{sc}}$ and $\lambda,\mu>0$, the minimum $M_{m,n}(\lambda,\mu)$ is attained.
\end{proposition}

\begin{proof}
By Propositions~\ref{prop:AB_continuous} and~\ref{prop:AB_limits}, the length function is continuous and
\[
    \mathcal L_{m,n}(E;\lambda,\mu)\to\infty
\]
as $E\downarrow0$ and as $E\to\infty$. Continuity therefore gives a minimum on a compact subinterval of $(0,\infty)$.
\end{proof}

Define
\[
    \Pi_{m,n}(L)
    :=
    \{(\lambda,\mu)\in(0,\infty)^2:
       \exists E>0,\ \mathcal L_{m,n}(E;\lambda,\mu)=L\}.
\]

\begin{theorem}[Threshold characterization]
\label{thm:threshold}
For every $(m,n)\in\mathcal I_{\mathrm{sc}}$,
\[
    \Pi_{m,n}(L)
    =
    \{(\lambda,\mu)\in(0,\infty)^2:
       M_{m,n}(\lambda,\mu)\le L\}.
\]
\end{theorem}

\begin{proof}
If $(\lambda,\mu)\in\Pi_{m,n}(L)$, then some $E_0>0$ satisfies
$\mathcal L_{m,n}(E_0;\lambda,\mu)=L$, and hence
$M_{m,n}(\lambda,\mu)\le L$.

Conversely, let $E_*$ be a minimizing energy. If
$M_{m,n}(\lambda,\mu)=L$, then $E_*$ itself satisfies the length equation. If $M_{m,n}(\lambda,\mu)<L$, endpoint divergence provides energies $0<E_-<E_*<E_+$ such that
\[
    \mathcal L_{m,n}(E_-;\lambda,\mu)>L,
    \qquad
    \mathcal L_{m,n}(E_+;\lambda,\mu)>L.
\]
Continuity and the intermediate value theorem give one root in $(E_-,E_*)$ and another in $(E_*,E_+)$. In particular, at least one energy satisfies the length equation, so $(\lambda,\mu)\in\Pi_{m,n}(L)$.
\end{proof}

\begin{theorem}[Full spectrum characterization]
\label{thm:full_spectrum}
Let
\[
    \Sigma_{p,q}(L)
    :=
    \{(\lambda,\mu)\in\R^2:
      \eqref{eq:model}\text{ has a nontrivial weak solution}\}.
\]
Then
\[
    \Sigma_{p,q}(L)
    =
    \mathcal P_+(L)
    \cup
    \mathcal P_-(L)
    \cup
    \bigcup_{(m,n)\in\mathcal I_{\mathrm{sc}}}\Pi_{m,n}(L).
\]
\end{theorem}

\begin{proof}
Let $(\lambda,\mu)$ belong to the spectrum and let $u$ be a nontrivial
solution. If $u$ is nonnegative or nonpositive, Propositions~
\ref{prop:positive_one_hump} and~\ref{prop:negative_one_hump} show that
\[
    (\lambda,\mu)\in
    \mathcal P_+(L)\cup\mathcal P_-(L).
\]

If $u$ changes sign, Corollary~\ref{cor:quadrant_reduction} gives
$\lambda,\mu>0$. Let $m$ and $n$ be the numbers of positive and negative
nodal intervals. Theorem~\ref{thm:nodal_reduction} gives
\[
    \mathcal L_{m,n}(E;\lambda,\mu)=L,
\]
so the parameter lies in $\Pi_{m,n}(L)$.

Conversely, Propositions~\ref{prop:positive_one_hump} and
\ref{prop:negative_one_hump} produce a one-sign solution for every
parameter in $\mathcal P_+(L)$ or $\mathcal P_-(L)$. If
$(\lambda,\mu)\in\Pi_{m,n}(L)$, choose an energy satisfying the length
equation and concatenate alternating positive and negative full humps.
Their endpoint derivative magnitudes are all $\eta(E)$, with signs that
match across each node. Thus the concatenated function and its flux are
continuous, and the piecewise equation holds globally in the weak sense.
If $u$ changes sign, Corollary~\ref{cor:quadrant_reduction} gives
$\lambda,\mu>0$. Let $m$ and $n$ be the numbers of positive and negative nodal intervals. Theorem~\ref{thm:nodal_reduction} gives
$\mathcal L_{m,n}(E;\lambda,\mu)=L$, so the parameter lies in
$\Pi_{m,n}(L)$.

Conversely, the full-hump construction produces a one-sign solution for every parameter in $\mathcal P_+(L)$ or $\mathcal P_-(L)$. If $(\lambda,\mu)\in\Pi_{m,n}(L)$, choose an energy satisfying the length equation and concatenate alternating positive and negative full humps. Their endpoint derivative magnitudes are all $\eta(E)$, with signs that match across each node. Thus the concatenated function and its flux are continuous, and the piecewise equation holds globally in the weak sense.
\end{proof}

\section{Boundary Formulas and Asymptotes}
\label{sec:boundary}

For fixed $\lambda>0$, define
\[
    \mathcal E_{m,\lambda}
    :=
    \{E>0:L>2m\lambda^{-1/p}A(E)\}.
\]
If this set is nonempty, define
\begin{equation}\label{eq:mu_star}
    \mu^*_{m,n}(\lambda)
    :=
    \inf_{E\in\mathcal E_{m,\lambda}}
    \left[
      \frac{2nB(E)}{L-2m\lambda^{-1/p}A(E)}
    \right]^q;
\end{equation}
otherwise set $\mu^*_{m,n}(\lambda)=+\infty$.

\begin{theorem}[Boundary formula]
\label{thm:boundary_formula}
For every $(m,n)\in\mathcal I_{\mathrm{sc}}$,
\[
    (\lambda,\mu)\in\Pi_{m,n}(L)
    \quad\Longleftrightarrow\quad
    \mu\ge\mu^*_{m,n}(\lambda).
\]
Whenever $\mathcal E_{m,\lambda}\ne\varnothing$, the infimum in
\eqref{eq:mu_star} is attained.
\end{theorem}

\begin{proof}
By Theorem~\ref{thm:threshold}, membership is equivalent to the existence of $E>0$ such that
\[
    2m\lambda^{-1/p}A(E)+2n\mu^{-1/q}B(E)\le L.
\]
For an admissible $E$, this is equivalent to
\[
    \mu\ge
    \left[
      \frac{2nB(E)}{L-2m\lambda^{-1/p}A(E)}
    \right]^q.
\]

It remains to justify the boundary value. If $\mathcal E_{m,\lambda}$ is nonempty, strict monotonicity of $A$ gives $\mathcal E_{m,\lambda}=(E_0,\infty)$ for a unique $E_0>0$. The displayed function tends to $+\infty$ as $E\downarrow E_0$ because its denominator tends to zero, and it tends to $+\infty$ as $E\to\infty$ because $B(E)\to\infty$.
It therefore attains its minimum. At the minimizing energy, equality gives a point of $\Pi_{m,n}(L)$; for larger $\mu$, the length is below $L$ at that energy and Theorem~\ref{thm:threshold} applies.
\end{proof}

Define
\[
    \lambda_m^*:=\left(\frac{m\pi_p}{L}\right)^p,
    \qquad
    \mu_n^*:=\left(\frac{n\pi_q}{L}\right)^q.
\]

\begin{theorem}[Domain, monotonicity, and asymptotes]
\label{thm:boundary_asymptotes}
Let $(m,n)\in\mathcal I_{\mathrm{sc}}$. Then:
\begin{enumerate}[label=\textup{(\roman*)}]
\item $\mu^*_{m,n}(\lambda)<\infty$ if and only if
$\lambda>\lambda_m^*$;
\item $\lambda\mapsto\mu^*_{m,n}(\lambda)$ is nonincreasing on
$(\lambda_m^*,\infty)$;
\item $\mu^*_{m,n}(\lambda)\to\infty$ as $\lambda\downarrow\lambda_m^*$;
\item $\mu^*_{m,n}(\lambda)\to\mu_n^*$ as $\lambda\to\infty$.
\end{enumerate}
\end{theorem}

\begin{proof}
The admissible set is nonempty exactly when
\[
    2m\lambda^{-1/p}\inf_{E>0}A(E)<L.
\]
Since $\inf A=\pi_p/2$, this is equivalent to
$\lambda>\lambda_m^*$.

If $\lambda_m^*<\lambda_1<\lambda_2$, then $\mathcal E_{m,\lambda_1}\subset\mathcal E_{m,\lambda_2}$ and the denominator in \eqref{eq:mu_star} is larger for $\lambda_2$. Taking infima proves monotonicity.

For the vertical asymptote, let $E_0(\lambda)$ be defined by
$L=2m\lambda^{-1/p}A(E_0)$. Then $E_0(\lambda)\to\infty$ as
$\lambda\downarrow\lambda_m^*$. Since $B$ is increasing,
\[
    \mu^*_{m,n}(\lambda)
    \ge
    \left[\frac{2nB(E_0(\lambda))}{L}\right]^q
    \to\infty.
\]

For the horizontal asymptote, $B(E)>\pi_q/2$ and the denominator in
\eqref{eq:mu_star} is below $L$, so $\mu^*_{m,n}(\lambda)\ge\mu_n^*$.

For the reverse bound, use the fact that $A$ is continuous, strictly decreasing, and maps $(0,\infty)$ onto $(\pi_p/2,\infty)$. For every sufficiently large $\lambda$, let $E_\lambda>0$ be the unique energy satisfying
\[
    A(E_\lambda)=\lambda^{1/(2p)}.
\]
Since $A(E_\lambda)\to\infty$, the endpoint behavior of $A$ implies
$E_\lambda\downarrow0$. Moreover,
\[
    \lambda^{-1/p}A(E_\lambda)=\lambda^{-1/(2p)}\longrightarrow0,
\]
so $E_\lambda\in\mathcal E_{m,\lambda}$ for all sufficiently large $\lambda$. Also $B(E_\lambda)\to\pi_q/2$. Substitution into \eqref{eq:mu_star} gives
\[
\begin{aligned}
    \limsup_{\lambda\to\infty}\mu^*_{m,n}(\lambda)
    &\le
    \lim_{\lambda\to\infty}
    \left[
      \frac{2nB(E_\lambda)}
      {L-2m\lambda^{-1/p}A(E_\lambda)}
    \right]^q\\
    &=\left(\frac{n\pi_q}{L}\right)^q
    =\mu_n^*.
\end{aligned}
\]
Combining the lower and upper bounds proves the limit.
\end{proof}

\section{Numerical Illustration of the Spectral Regions}
\label{sec:numerical_region}

We illustrate the difference between the first sign-changing Fu\v{c}ik
branch for the homogeneous $p$-Laplacian and the corresponding spectral
region for \eqref{eq:model}. The numerical plots illustrate the characterization in
Theorem~\ref{thm:threshold} and the boundary description in
Theorem~\ref{thm:boundary_formula}; they are not used in the proofs.
Throughout this section we take
\[
    p=3,\qquad q=2,\qquad L=1,
\]
and consider the nodal class $(m,n)=(1,1)$, consisting of one positive
and one negative nodal interval.

\begin{figure}[htbp]
    \centering
    \begin{minipage}[t]{0.49\textwidth}
        \centering
        \includegraphics[width=\linewidth]{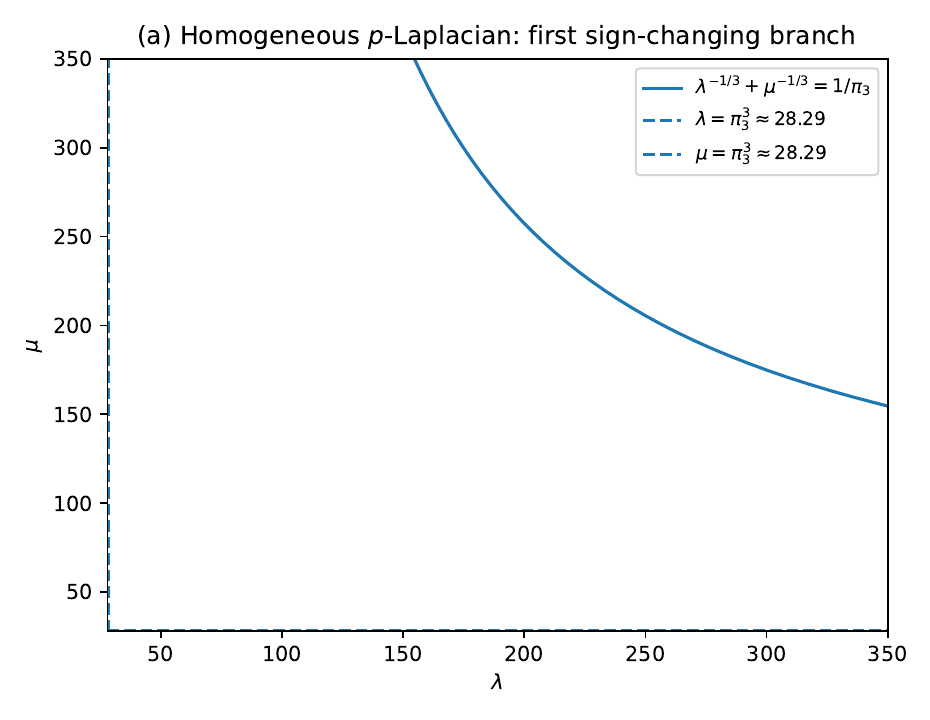}
    \end{minipage}\hfill
    \begin{minipage}[t]{0.49\textwidth}
        \centering
        \includegraphics[width=\linewidth]{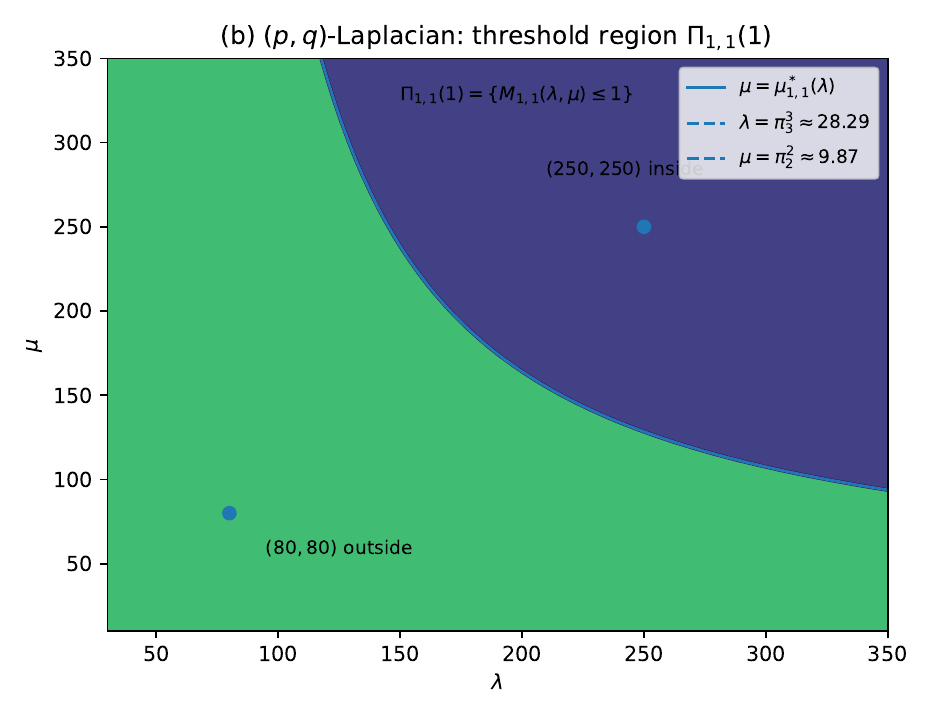}
    \end{minipage}
    \caption{Comparison of the first sign-changing nodal class for
    $L=1$. Panel (a) shows the one-dimensional Fu\v{c}ik curve
    $\lambda^{-1/3}+\mu^{-1/3}=1/\pi_3$ for the homogeneous
    $p$-Laplacian with $p=3$. Panel (b) shows the threshold region
    $\Pi_{1,1}(1)$ for the nonhomogeneous $(p,q)$-Laplacian with
    $(p,q)=(3,2)$. In panel (b), the region above the solid boundary
    $\mu=\mu^*_{1,1}(\lambda)$ satisfies $M_{1,1}(\lambda,\mu)\le1$;
    the region below it satisfies $M_{1,1}(\lambda,\mu)>1$.
    The marked points show parameter values for which a solution with
    one positive and one negative hump exists or does not exist.
    A point below the boundary may still give a one-sign solution.}
    \label{fig:first_nodal_curve_vs_region}
\end{figure}

\subsection*{The homogeneous $p$-Laplacian branch}

For comparison, consider
\[
    -(\phi_3(u'))'
    =\lambda(u^+)^2-\mu(u^-)^2,
    \qquad u(0)=u(1)=0.
\]
For the homogeneous $p$-Laplacian, the first sign-changing branch
is a curve \cite{CuestaFigueiredoGossez1999}. As in Section~10.3,
the lengths of the positive and negative humps are independent of
their amplitude. For $p=3$ and $L=1$, one hump of each sign gives
\begin{equation}\label{eq:hom_first_nodal_curve}
    \pi_3\lambda^{-1/3}+\pi_3\mu^{-1/3}=1,
    \qquad\text{or equivalently}\qquad
    \lambda^{-1/3}+\mu^{-1/3}=\frac{1}{\pi_3}.
\end{equation}
This is the curve shown in panel (a).

\subsection*{The nonhomogeneous $(p,q)$-Laplacian region}

For the asymmetric problem \eqref{eq:model} with $(p,q)=(3,2)$, the
kinetic-energy function becomes
\[
    h(\xi)=\frac{2}{3}\xi^3+\frac12\xi^2,
    \qquad \eta=h^{-1}.
\]
The parameter-free time maps are
\[
\begin{aligned}
    A(E)&=3^{1/3}E^{1/3}
      \int_0^1\frac{dt}{\eta\bigl(E(1-t^3)\bigr)},\\
    B(E)&=2^{1/2}E^{1/2}
      \int_0^1\frac{dt}{\eta\bigl(E(1-t^2)\bigr)}.
\end{aligned}
\]
The length condition for one positive and one negative hump now
depends on the common energy $E>0$:
\begin{equation}\label{eq:numerical_first_length}
    \mathcal L_{1,1}(E;\lambda,\mu)
    =2\lambda^{-1/3}A(E)
     +2\mu^{-1/2}B(E)=1.
\end{equation}
We evaluate the two time maps numerically and minimize
the length expression on the left of
\eqref{eq:numerical_first_length} over positive energies to obtain
\[
    M_{1,1}(\lambda,\mu)
    =\min_{E>0}\left[
       2\lambda^{-1/3}A(E)+2\mu^{-1/2}B(E)
    \right].
\]
Theorem~\ref{thm:threshold} states that
\[
    \Pi_{1,1}(1)
    =\{(\lambda,\mu)\in(0,\infty)^2:
       M_{1,1}(\lambda,\mu)\le1\}.
\]
The filled set in panel (b) therefore represents the $(1,1)$-nodal
spectral region. Its solid boundary is described by
Theorem~\ref{thm:boundary_formula} and, in this example, is obtained
from
\begin{equation}\label{eq:numerical_mu_star}
    \mu^*_{1,1}(\lambda)
    =\inf_{\substack{E>0\\2\lambda^{-1/3}A(E)<1}}
      \left[
        \frac{2B(E)}{1-2\lambda^{-1/3}A(E)}
      \right]^2.
\end{equation}
The curve $\mu=\mu^*_{1,1}(\lambda)$ is the boundary of this
region. For each admissible $\lambda$, every
$\mu\ge\mu^*_{1,1}(\lambda)$ gives a solution with one positive
and one negative hump. Thus panel (a) shows a spectral curve,
whereas panel (b) shows a spectral region.

The two marked points illustrate this distinction.
Numerically, $M_{1,1}(80,80)$ is approximately $1.38$, which is
strictly greater than $L=1$. Hence $(80,80)$ does not admit a
sign-changing solution with exactly one positive and one negative
hump. In contrast, $M_{1,1}(250,250)$ is approximately $0.89<1$.
By Theorem~\ref{thm:threshold}, there is a positive energy for which
\eqref{eq:numerical_first_length} holds. The corresponding positive
and negative full humps can then be joined to give a Dirichlet
solution of this nodal type. Since the length function tends to $+\infty$ at both ends of the
energy interval, there are at least two such energies when
$M_{1,1}(\lambda,\mu)<1$.

A parameter pair outside $\Pi_{1,1}(1)$ may still belong to the
full spectrum $\Sigma_{3,2}(1)$. For example,
$(80,80)$ lies in both one-sign regions $\mathcal P_+(1)$ and
$\mathcal P_-(1)$, because $80>\pi_3^3$ and $80>\pi_2^2$.
Thus $(80,80)$ gives both a positive solution and a negative solution,
even though it does not give a solution with one hump of each sign.

Finally, Theorem~\ref{thm:boundary_asymptotes} gives the asymptotic
geometry of the boundary:
\[
\begin{aligned}
    \lambda_1^*&=\pi_3^3\approx28.29,
      &\mu^*_{1,1}(\lambda)&\longrightarrow+\infty
        &&\text{as }\lambda\downarrow\lambda_1^*,\\
    \mu_1^*&=\pi_2^2=\pi^2\approx9.87,
      &\mu^*_{1,1}(\lambda)&\longrightarrow\mu_1^*
        &&\text{as }\lambda\to\infty.
\end{aligned}
\]
Panel (b) shows an intermediate portion of this boundary.
The asymptotes lie outside the displayed plotting window.
These numerical results illustrate the boundary formula and limits
proved in Section~12.

\nocite{*}
\bibliographystyle{plain}
\bibliography{references}

\end{document}